\documentclass[12pt]{amsart}

\usepackage{amssymb}
\usepackage{amsmath}
\usepackage{amscd}
\usepackage{multicol}
\usepackage[dvipdfmx]{graphicx}

\newtheorem{theorem}{Theorem}

\newtheorem{proposition}[theorem]{Proposition}

\newtheorem{claim}{Claim}

\theoremstyle{definition}
\newtheorem{definition}{Definition}
\newtheorem{example}{Example}
\newtheorem{problem}{Problem}

\theoremstyle{remark}
\newtheorem{remark}{Remark}

\newenvironment{namelist}[1]{%
\begin{list}{}
  {
   \settowidth{\labelwidth}{#1}
   \setlength{\leftmargin}{2.5\labelwidth}}
}{%
\end{list}}

\def\dfrac#1#2{{\displaystyle\frac{#1}{#2}}}

\begin{document}

\title[birational transformations belonging to Galois points]{Extendable birational transformations belonging to Galois points}

\author{Kei Miura}
\address{Department of Mathematics, National Institute of Technology, Ube College, 
2-14-1 Tokiwadai, Ube, Yamaguchi 755-8555, Japan}
\email{kmiura@ube-k.ac.jp}




\begin{abstract}
We study birational transformations belonging to Galois points. 
Let $P$ be a Galois point for a plane curve $C$ and $G_P$ be a  
Galois group at $P$. 
Then an element of $G_P$ induces a birational transformation of $C$. 
In general, it is difficult to determine when this birational transformations can be extended to a Cremona 
(or projective) transformation. 
In this article, we shall prove that if the Galois group is isomorphic to the cyclic group of 
order three, 
then any element of the Galois group has an expression as a de Jonqui\`{e}res transformation. 
In particular, they can be extended to Cremona transformations. 
\end{abstract}

\maketitle


\section{Introduction}
Let $k$ be the field of complex numbers ${\mathbb C}$. 
We fix it as the ground field of our discussion. 
First, we briefly recall the notion of Galois points. 
Let $C$ be an irreducible (possibly singular) curve in ${\mathbb P}^2$ of degree $d$ $(d \geq 3)$, 
and $k(C)$ be the function field of $C$. 
Taking a point $P$ of ${\mathbb P}^2$, we consider a projection 
$\pi_P : C \dashrightarrow {\mathbb P}^1$, 
which is the restriction of the projection 
${\mathbb P}^2 \dashrightarrow {\mathbb P}^1$ with the center $P$. 
Then, we obtain a field extension induced by $\pi_P$, 
i.e., $\pi_P^* : k({\mathbb P}^1) \hookrightarrow k(C)$. 
We remark that $\deg \pi_P^* = d - m_P$, where $m_P$ is the multiplicity of $C$ at $P$ 
(we put $m_P = 0$, if $P \notin C$). 
By putting $K_P = \pi_P^*(k({\mathbb P}^1))$, we have the following definition.

\begin{definition}
The point $P$ is called a {\em Galois point} for $C$ if the field extension $k(C) / K_P$ 
is Galois. 
Then we put $G_P = {\rm Gal} (k(C)/ K_P)$ and call a {\em Galois group at $P$}. 
\end{definition}

The notion of Galois point for $C$ was introduced by Yoshihara (cf. \cite{m-y1}) 
to study the structure of the field extension $k ( C ) / k$ from geometrical viewpoint.

We have studied Galois points for some cases (cf. \cite{miura1}, \cite{miura4}, \cite{miura5}, 
\cite{m-y1}, \cite{yoshi}, \cite{yoshi2}, \cite{yoshi3}, etc.). 
For example, we determined the number of the Galois points for $C$ 
and found the characterization of $C$ having Galois points. 
On the other hand, in the case where $P$ is not a Galois point, 
we determined the monodromy group and constructed the variety 
corresponding to the Galois closure.

\medskip

Hereafter, we suppose that $P$ is a Galois point for $C$. 
The objective of this article is to investigate the action of $G_P$ on $C$. 
Specifically, our problem is stated as follows. 
Let $\sigma$ be an element of $G_P$. 
Then, $\sigma$ induces a birational transformation of $C$ over ${\mathbb P}^1$, 
which we denote by the same letter $\sigma$. 
We call $\sigma$ a {\em birational transformation belonging to Galois point} $P$.

\begin{problem}
Study the properties of $\sigma$. 
\end{problem}

Concerning this problem, we first mention the following theorem. 

\medskip

\noindent
{\bf Theorem A} (Yoshihara \cite{yoshi}, \cite{yoshi2}){\bf .} 
{\em 
Let $V$ be a smooth hypersurface in ${\mathbb P}^n$ of degree $d$ $(d \geq 4)$. 
Then, every element $\sigma \in G_P$ is a restriction of a projective transformation of ${\mathbb P}^n$. 
In other words, there exists $\widetilde{\sigma} \in {\rm PGL} (n + 1, k)$ such that $\widetilde{\sigma} |_V = \sigma$. 
Furthermore, $G_P$ is a cyclic group. 
}

\medskip

The conditions of smoothness and $d \geq 4$ are both required for the theorem to hold. 
However, in the case where $C$ has a singular point, 
$G_P$ is not necessarily cyclic nor can its element always be extended to a projective 
transformation (cf.\ \cite{miura5}, \cite{yoshi}). 
In particular, in \cite{miura5}, we constructed examples such that 
$G_P$ could not be extended to a Cremona transformation 
by applying to Iitaka's L-minimality. 
It seems to be an important problem to know if an element can be extended to a birational 
transformation of ${\mathbb P}^2$. 
But, this problem has remained open for several years. 
Furthermore, Yoshihara \cite{yoshi3} gave a negative answer in general by presenting 
several examples.

Although a special case, we have the following theorem for a smooth cubic hypersurface. 
Then we note that any point on a cubic hypersurface becomes a Galois point such that 
the Galois group is isomorphic to the cyclic group of order two.

\noindent
{\bf Theorem B} (Miura \cite{miura6}){\bf .} 
{\em 
Suppose that $V$ is a smooth cubic hypersurface in ${\mathbb P}^n$. 
Then the birational transformation belonging to $P \in V$ is a restriction of a projective transformation of ${\mathbb P}^n$ 
if and only if $P$ is an Eckardt point. 
On the contrary, if not, then every element of the Galois group at $P \in V$ can be expressed as a 
birational transformation of ${\mathbb P}^n$. 
}

\medskip

In this article, we study the problem for the case where $G_P$ is isomorphic to 
${\mathbb Z}_3$; the cyclic group of order three. In this case, we can prove that an element of $G_P$ can be extended to a birational 
transformation of ${\mathbb P}^2$. 
Precisely, our results are stated as follows.

\begin{theorem}\label{main}
Let $C$ be an irreducible (possibly singular) plane curve of degree $d$ $(d \geq 3)$. 
Suppose that $P$ is a Galois point for $C$ and its Galois group $G_P$ is isomorphic to 
${\mathbb Z}_3$. 
Then, any element of $G_P$ has an expression as a de Jonqui\`{e}res transformation. 
In particular, any element of $G_P$ can be extended to a Cremona transformation of ${\mathbb P}^2$. 
\end{theorem}

\begin{remark}
As long as $G_P \cong {\mathbb Z}_3$, the point $P$ can be $P \in C$ or $P \in {\mathbb P}^2 \setminus C$. 
A Galois point $P$ is called {\em inner} (resp. {\em outer}) if $P \in C$ (resp. $P \in {\mathbb P}^2 \setminus C$). 
\end{remark}

\medskip

The Cremona group ${\rm Bir} ({\mathbb P}^2)$ is a classical object in algebraic geometry. 
Hence, there are so many results on ${\rm Bir} ({\mathbb P}^2)$. 
However, we have obtained very few results on birational transformations belonging to Galois points (e.g.,  \cite{miura4}, \cite{miura5}, \cite{yoshi3}). 
In particular, it is difficult to determine when a birational transformation extends to a Cremona transformation.


\section{de Jonqui\`{e}res transformations}
In this section, by referring to Pan and Simis \cite{pansim}, we recall the de Jonqui\`{e}res transformations. 
From \cite{pansim},  
a de Jonqui\`{e}res map is a plane Cremona map $\frak F$ of degree $d \geq 2$ satisfying any one of 
the following equivalent conditions:

\begin{namelist}{111}
 \item[{\rm (i)}] $\frak F$ has homaloidal type $(d ; d - 1, 1^{2d - 2})$. 
 \item[{\rm (ii)}] There exists a point $o \in {\mathbb P}^2$ such that the restriction of $\frak F$ to a general line passing 
 through $o$ maps it birationally to a line passing through $o$. 
 \item[{\rm (iii)}] Up to projective coordinate change (source and target), $\frak F$ is defined by $d$-forms $\left\{ q x_0, q x_1, f \right\}$
 such that $f$, $q \in k [x_0, x_1, x_2]$ are relatively prime $x_2$-monoids, one of which at least has degree $1$ in $x_2$. 
\end{namelist}

Let $J_o ({\mathbb P}^2) \subset {\rm Bir} ({\mathbb P}^2)$ be the subgroup of all de Jonqui\`{e}res transformation 
with center $o$. 
Putting $o = (0:0:1)$, we obtain an explicit form of de Jonqui\`{e}res transformation by the third alternative 
as follows:
$$(x_0 : x_1 : x_2) \mapsto (q x_0 : q x_1 : f), $$
where $f = a(x_0, x_1)x_2 + b(x_0, x_1)$, $q = c(x_0, x_1)x_2 + d(x_0, x_1)$ 
such that $a, b, c, d \in k[x_0, x_1]$ and $ad - bc \ne 0$. 
We note that the fraction $f/q = (a(x_0, x_1)x_2 + b(x_0, x_1))/(c(x_0, x_1)x_2 + d(x_0, x_1))$ defines a 
M\"{o}bius transformation in the variable $x_2$ over the function field of ${\mathbb P}^1$.

Let $H$ be a line defined by $x_2 = 0$. 
Note that the projective space of line passing through $o$ is identified with $H$. 
By definition, we infer that de Jonqui\`{e}res transformation $\frak F \in J_o ({\mathbb P}^2)$ induces a 
birational map $H \dashrightarrow H$. 
By identifying $H = {\mathbb P}^1$, we obtain a group homomorphism
$$\rho : J_o ({\mathbb P}^2) \rightarrow {\rm Bir} ({\mathbb P}^1).$$

Now, we recall the following two objects:
\begin{itemize}
 \item The element $\frak{f}_{a,b,c,d} : = (a x_2 + b) / (c x_2 + d) \in k [x_0, x_1] (x_2).$
 \item The rational map $\frak{F}_{a,b,c,d}: {\mathbb P}^2 \dashrightarrow {\mathbb P}^2$ defined by 
 $((c x_2 + d)x_0 : (c x_2 + d)x_1 : a x_2 + b). $
\end{itemize}

Then, Pan and Simis obtained the following theorem.

\medskip

\noindent
{\bf Theorem C} (Pan and Simis \cite{pansim}, Proposition 1.2){\bf .}  
{\em 
Let $\rho : J_o ({\mathbb P}^2) \rightarrow {\rm Bir} ({\mathbb P}^1)$ be as above. 
Fix $o = (0:0:1)$ and $H = \{ x_2 = 0 \}$. 
We denote $K = k(H)$ the function field of $H$. 
Then, 

\begin{namelist}{11}
 \item[{\rm (i)}] A Cremona map $\frak F \in {\rm Bir} ({\mathbb P}^2)$ belongs $J_o ({\mathbb P}^2)$ if and only if as a rational map 
 ${\mathbb P}^2 \dashrightarrow {\mathbb P}^2$ it has the form $(q g_0 : q g_1 : f)$, where $(g_0 : g_1)$ defines a Cremona map of $H \cong {\mathbb P}^1$ 
 and $q$, $f \in k[x_0, x_1, x_2]$ are relatively prime $x_2$-monoids at least one of which has positive $x_2$-degree.   
 \item[{\rm (ii)}] The group ${\rm PGL} (2, K)$ can be identified with the M\"{o}bius group whose elements have the form $\frak{f}_{a,b,c,d}$. 
 \item[{\rm (iii)}] The map $\frak{f}_{a,b,c,d} \mapsto \frak{F}_{a,b,c,d}$ is an injective group homomorphism $\psi : {\rm PGL} (2, K) \hookrightarrow J_o ({\mathbb P}^2)$. 
 \item[{\rm (iv)}] ${\rm Im} ( \psi ) = {\rm Ker} ( \rho )$; in particular $\frak{F}_{a,b,c,d}$ maps a general line passing through the point $o = (0:0:1)$ birationally to itself. 
\end{namelist}
}

\medskip


\section{Proof of Theorem 1}
We prove our main theorem. 
Our strategy is as follows. 
Suppose that $P$ is a Galois point for $C$ and its Galois group $G_P$ is isomorphic to ${\mathbb Z}_3$, 
that is $G_P = {\rm Gal} (k(C) / K_P) \cong {\mathbb Z}_3$. 
Then we shall prove that $G_P \ni \sigma$ has a representation as an element of ${\rm PGL} (2, K_P)$. 
Hence, by Theorem C, we get our result.

Here, we recall our notation. 
Let $f(x, y) = 0$ be the defining equation of $C$ and $P = (0, 0)$. 
Let $l_t$ be the line $y = t x$. 
Then we may assume that the projection $\pi_P$ is defined as $\pi_P (C \cap l_t) = t$. 
Moreover we put $\widehat{f}(x, t) = f (x, tx) / x^m$, where $m = m_P$ is the multiplicity of $C$ 
at $P$. 
Then $k(C) = k(x, t)$ and $K_P = k(t)$, the extension $k(C) / K_P$ is given by $\widehat{f}(x, t) = 0$.

Now, more generally, suppose that $G_P$ is isomorphic to ${\mathbb Z}_n$, the cyclic group of order $n$ 
for some integer $n$ ($n \geq 3$).
Then, by our assumption, we may take $k(C) = 2K_P (\sqrt[n]{q})$ for some $q \in K_P$. 
Then, it is clear that $\sigma (\sqrt[n]{q}) = \zeta \sqrt[n]{q}$ for some $\sigma$ satisfying 
$\langle \sigma \rangle = G_P$, 
where $\zeta$ is a primitive $n$-th root of unity. 
Under this situation, we prove the following proposition.

\begin{proposition}\label{pro}
Let $x$ be an element of $K_P (\sqrt[n]{q})$ such that 
$x = c_0 + c_1 \sqrt[n]{q} + c_2 (\sqrt[n]{q})^2 + \cdots + c_{n-1} (\sqrt[n]{q})^{n-1}$, 
where $c_0$, $c_1$, $\ldots$, $c_{n -1 } \in K_P$.   
If $c_1$, $c_2$, $\ldots$, $c_{n-1}$ forms a geometric sequence, 
then $\sigma (x)$ is expressed as a M\"{o}bius transformation $M (x)$ for some $M \in {\rm PGL}(2, K_P)$.  
\end{proposition}

\begin{proof}
By the property of a geometric series, 
$$c_1 \sqrt[n]{q} + c_2 (\sqrt[n]{q})^2 + \cdots + c_{n-1} (\sqrt[n]{q})^{n-1}$$
$$= \frac{c_1 \sqrt[n]{q} (1 - (\frac{c_2}{c_1} \sqrt[n]{q})^{n - 1})}{1 - \frac{c_2}{c_1} \sqrt[n]{q}}$$
$$= \frac{c_1 \sqrt[n]{q} - c_1 \frac{c_2^{n-1}}{c_1^{n-1}} q}{1 - \frac{c_2}{c_1} \sqrt[n]{q}}$$
$$= \frac{c_1^{n - 1} \sqrt[n]{q} - c_2^{n-1} q}{c_1^{n - 2} - c_1^{n - 3} c_2 \sqrt[n]{q}}.$$
Therefore, 
$$x = c_0 + \frac{c_1^{n - 1} \sqrt[n]{q} - c_2^{n-1} q}{c_1^{n - 2} - c_1^{n - 3} c_2 \sqrt[n]{q}}$$
$$= \frac{(c_1^{n - 1} - c_0 c_1^{n - 3} c_2) \sqrt[n]{q} + c_0 c_1^{n - 2} - c_2^{n - 1}q}
{- c_1^{n - 3} c_2 \sqrt[n]{q} + c_1^{n - 2}},$$
so that 
$$\sigma (x) = \frac{(c_1^{n - 1} - c_0 c_1^{n - 3} c_2) \zeta \sqrt[n]{q} + c_0 c_1^{n - 2} - c_2^{n - 1}q}
{- c_1^{n - 3} c_2 \zeta \sqrt[n]{q} + c_1^{n - 2}}. $$
Conversely, 
$$\sqrt[n]{p} = \frac{c_1^{n - 2} x - c_0 c_1^{n - 2} + c_2^{n - 1}q}
{c_1^{n - 3} c_2 x + c_1^{n - 1} - c_0 c_1^{n - 3} c_2}.$$
Hence we infer that $\sigma (x)$ is expressed as a M\"{o}bius transformation 
$M (x)$ for some $M \in {\rm PGL}(2, K_P)$.
\end{proof}

In particular, $\sigma$ can be extended to a Cremona transformation by Theorem C. 
In the case where $n = 3$, then $c_1$ and $c_2$ always forms a geometric sequence. 
Therefore, we obtain our main result. 


\section{Examples}
We give some examples.

\begin{example} 
Let $C$ be a curve defined by
$5 x^{12} - 16 x^{13} + 18 x^{14} - 8 x^{15} + x^{16} - 6 x^8 y^4 + 12 x^9 y^4 - 6 x^{10} y^4 + 4 x^4 y^8 - 4 x^5 y^8 - y^{12} = 0$. 
In fact, $C$ is birationally equivalent to the quartic Fermat curve $x^4 + y^4 + 1 = 0$, 
since $k(C) = k(t) (\sqrt[4]{t^4 + 1})$. 
Indeed, we can obtain the defining equation of $C$ by 
$x = 2 + q + q^2 + q^3$, where $q = \sqrt[4]{t^4 + 1}$ and 
$t = y / x$. 
The point $P = (0, 0) \in C$ is a singular point of multiplicity $12$, 
so the degree of $\pi_P$ is $16 - 12 = 4$. 
By the construction, it is clear that $P$ is a Galois point and its Galois group $G_P$ is 
isomorphic to the cyclic group of order four. 
We remark that the above $x$ satisfies the condition of Proposition \ref{pro}. 
Actually, for $\sigma \in G_P$, 
$$\sigma (x) = \frac{- \zeta q + 1 - t^4}{- \zeta q + 1}. $$
Hence we infer that $\sigma$ is represented as the following M\"{o}bius transformation 
$$M = \left(\begin{array}{cc}
   - \zeta & 1 - t^4  \\
   - \zeta & 1         
  \end{array}\right) \in {\rm PGL}(2, K_P).$$
In particular, $\sigma$ can be extended to a Cremona transformation by Theorem C. 
\end{example}

\begin{example}
Suppose that $C$ is a plane quartic curve with a simple cusp of multiplicity three. 
Then, it is well known that $C$ is projectively equivalent 
to one of the following (cf.\ \cite{namba1}):

\begin{namelist}{(a)}
 \item[{\rm (a)}] $X^4 - X^3 Y + Y^3 Z = 0$,
 \item[{\rm (b)}] $X^4 - Y^3 Z = 0$.
\end{namelist}

In the case of (a), $C$ has two flexes of order one. 
Indeed $Q_1 = (0 : 1 : 0)$ and $Q_2 = (8 : 16 : 3)$ are flexes. 
Then, the tangent lines at these flexes meet $C$ at $P_1 = (1 : 1 : 0)$ and 
$P_2 = (8 : -16 : 3)$, respectively. 
We can easily check that $\pi_{P_i}$ gives a totally ramified triple covering. 
Hence $P_i$ is a Galois point for $C$ $(i= 1, 2)$.
By the condition of flexes of $C$, there is no other inner Galois point. 
In the case of (b), we have a Galois point $P_3 = (0 : 1 : 0)$, which is a flex of order two. 
Since there is no more flex on $C$, there is no other inner Galois point.

We remark that the curve of type ${\rm (b)}$ has an outer Galois point 
$P_4 = (1 : 0 : 0) \in {\mathbb P}^2 \setminus C$, 
its Galois group is isomorphic to the cyclic group of order four. 
We can easily check that $P_4$ is a unique outer Galois point for $C$ of type ${\rm (b)}$.  
That is, the curve of type ${\rm (a)}$ has two inner Galois points, and 
the curve of type ${\rm (b)}$ has an inner Galois point and an outer Galois point.

The Galois groups at $P_i$ are all isomorphic to the cyclic group of order three, ${\mathbb Z}_3$ 
$(i = 1, 2, 3)$. 
By setting $G_{P_i} = \langle \sigma_i \rangle$ $(i = 1, 2, 3)$, 
we shall show that the transformations $\sigma_1$, $\sigma_2$, and $\sigma_3$ are de Jonqui\`{e}res transformations.

First, we prove the case for type ${\rm (a)}$.

\begin{claim}
Suppose that $C$ is a curve of type ${\rm (a)}$. 
Then, there exists a linear automorphism $A$ of $C$ such that $A (P_1) = P_2$.   
\end{claim}

\begin{proof}
By putting $A = \left(
\begin{array}{ccc}
16  & -8  &  0  \\
0   & -16 &  0  \\
4   & -1  &  16 \\
\end{array}
\right)$, we can easily check that $A (C) = C$ and $A (P_1) = P_2$. 
\end{proof}

Therefore, it is sufficient to study on $\sigma_1$. 
We can obtain a concrete representation of $\sigma_1$ as follows.

\begin{claim}
$\sigma_1$ is represented as $(X : Y : Z) \mapsto
(X Y : Y ((\omega - 1) X + \omega Y) : Z ((\omega - 1) X + \omega Y))$, 
up to projective coordinate change, where $\omega$ is a primitive cubic root of unity.  
\end{claim}

\begin{proof}
By taking a suitable projective coordinate, we may assume that $P_1$ is translated to $P_1' = (1 : 0 : 0)$. 
Then, $\pi_{P_1'}$ is represented as $(X : Y : Z) \dashrightarrow (Y : Z)$ and 
$C$ is defined by $(X + Y)^3 Z - X^3 Y = 0$. 
By putting $x = X / Z$, $y = Y / Z$, we can obtain the field extension induced by $\pi_{P_1'}$ as
\begin{equation*}
 \begin{CD}
  k( C ) @= k(x, y)     \\
  \cup   @. \cup \\
  k({\mathbb P}^1)@= k(y)\\
 \end{CD}
\end{equation*}
where $k(x, y) / k(y)$ is given by the equation $(x + y)^3 - x^3 y = 0$. 
Then, we obtain $\left( \dfrac{x + y}{x} \right)^3 = y$, i.e., $\left( 1 + \dfrac{y}{x} \right)^3 = y$. 
We put $G_{P_1'} = \langle \sigma_1' \rangle$. 
Hence, we infer that $\sigma_1' \left( 1 + \dfrac{y}{x} \right) =  \omega \left( 1 + \dfrac{y}{x} \right)$. 
Thus, we conclude that $\sigma_1' (x) = \dfrac{y x}{\omega x - x + \omega y}$. 
Since $\sigma_1' (x : y: 1) = (\sigma_1' (x) : y : 1)$ and $x = X / Z$, $y = Y / Z$, 
we get the claim. 
We note that $\sigma_1'$ is conjugate to $\sigma_1$ up to projective coordinate change.  
\end{proof}

Next, we prove the case for type ${\rm (b)}$. 
By putting $x = X / Y$, $y = Z / Y$, we obtain the field extension $\pi_{P_3}^*$ as 
$k(x, y) / k(y)$, where $x^3 - y = 0$. 
Then, it is easy to see that $\sigma_3$ is represented as 
$\left(
\begin{array}{ccc}
\omega    & 0  &  0      \\
0         & 1  &  0      \\
0         & 0  &  \omega \\
\end{array}
\right)$.

Finally we mention $P_4$. 
Let $G_{P_4} = \langle \sigma_4 \rangle \cong {\mathbb Z}_4$. 
Then it is easy to check $\sigma_4$ is represented as 
$\left(
\begin{array}{ccc}
\eta      & 0  &  0      \\
0         & 1  &  0      \\
0         & 0  &  1 \\
\end{array}
\right)$, 
where $\eta$ is a primitive 4-th root of unity. 
Thus, we complete the proof of Example 2. 

\end{example}


\section{Problems}

Finally, we raise some problems. 

\begin{problem}
Let $P$ be a Galois point for $C$. 

 \begin{namelist}{111}
  \item[(1)] Find the condition when a birational transformation belonging to $P$ can be extended to 
  a Cremona transformation. 
  \item[(2)] Suppose that a birational transformation belonging to $P$ is extended to a Cremona transformation. 
  Then, is this a de Jonqui\`{e}res transformation?  
 \end{namelist}
\end{problem}


\section{Acknowledgements} 
The author are grateful to Dr. Mitsunori Kanazawa for helpful comments. 
The author was partially supported by JSPS KAKENHI Grant Number JP18K03230.


\end{document}